\documentclass[11pt,a4paper,reqno]{amsart}
\usepackage[T1]{fontenc}
\usepackage{lmodern}
\usepackage[margin= 1.24in]{geometry}
\usepackage{amsmath,amssymb,amsthm}
\usepackage{microtype}
\usepackage[colorlinks=true,linkcolor=blue,citecolor=blue,urlcolor=blue]{hyperref}
\usepackage[nameinlink,noabbrev]{cleveref}
\numberwithin{equation}{section}
\newtheorem{maintheorem}{Theorem}

\newtheorem{theorem}{Theorem}[section]
\newtheorem{proposition}[theorem]{Proposition}
\newtheorem{lemma}[theorem]{Lemma}
\newtheorem{corollary}[theorem]{Corollary}
\theoremstyle{remark}
\newtheorem{example}[theorem]{Example}
\newcommand{\R}{\mathbb R}
\newcommand{\RP}{\mathbb{RP}}
\newcommand{\C}{\mathcal C}
\newcommand{\dd}{\mathrm d}
\newcommand{\SL}{\mathrm{SL}}
\newcommand{\GL}{\mathrm{GL}}
\newcommand{\vol}{\operatorname{vol}}
\newcommand{\bdel}{\bar\partial}
\newcommand{\bdelt}{\bar\delta}
\newcommand{\bBox}{\overline\Box}
\newcommand{\cQ}{\mathcal Q}
\newcommand{\id}{\operatorname{Id}}
\newcommand{\grad}{\operatorname{grad}}
\newcommand{\rank}{\operatorname{rank}}
\title[Euler characteristic of convex projective $4$-manifolds]{Euler characteristic and flat cohomology\\of convex projective $4$-manifolds}
\author{Andrea Tamburelli}
\date{18 October 2026}

\begin{document}
\maketitle

\begin{abstract}
Let $M=\Omega/\Gamma$ be a closed properly convex real projective $4$-manifold.
We prove the Chern-Hopf-Thurston conjecture for this class of manifolds, i.e. $\chi(M)\geq0$. 
The argument transfers Akagawa's double-form
identity on the Cheng--Yau cone over $\Omega$ to a Bochner identity on $M$, from which we deduce that
$H^q(M;\Lambda^pE^*)=0$ whenever
$q<p(n+1-p)/(n+1)$, in arbitrary dimension $n$, for some flat bundle $E$ naturally attached to $M$.
In dimension four, this implies that 
the cohomology with coefficients in $\Lambda^2E^*$ concentrates in degree two.
\end{abstract}

\setcounter{tocdepth}{1}
\tableofcontents

\section*{Introduction}\label{sec:intro}

The Chern--Hopf--Thurston conjecture predicts that every closed
aspherical manifold of dimension $2m$ satisfies
\begin{equation}\label{eq:Hopf}
 (-1)^m\chi(M)\geq0.
\end{equation}
The conjecture has been established under some geometric,
complex-analytic, and topological assumptions. For closed real hyperbolic
manifolds, the predicted sign follows from the Gauss--Bonnet theorem,
and more generally it holds for closed locally symmetric manifolds
of noncompact type \cite{Olb02}. Gromov proved the strict inequality
for closed K\"ahler hyperbolic manifolds \cite{Gro91}, while
Cao--Xavier and Jost--Zuo established the non-strict inequality for
compact K\"ahler manifolds with nonpositive sectional curvature
\cite{CX01,JZ00}. In dimension four, the conjecture holds for closed
Riemannian manifolds with nonpositive sectional curvature
\cite{Chern}, closed aspherical complex surfaces \cite{JK93},
closed nonpositively curved piecewise Euclidean cubical manifolds
\cite{DO01}, and closed generalized Haken manifolds \cite{DE14}.
For closed aspherical manifolds whose fundamental groups contain
an infinite amenable normal subgroup, Cheeger--Gromov proved the
stronger conclusion $\chi(M)=0$ \cite{CG86}. More recently,
Deng--Wang established the conjecture for closed aspherical complex
projective manifolds whose fundamental groups admit a linear
representation with finite kernel over a field of arbitrary
characteristic \cite{DengWang}.

In this paper we establish the conjecture for closed
properly convex real projective $4$-manifolds. A properly
convex projective $n$-manifold is a quotient $M=\Omega/\Gamma$, where
$\Omega\subset\RP^n$ is an open convex set bounded in an affine
chart, and $\Gamma$ acts freely, properly discontinuously, and
cocompactly by projective transformations. The domain $\Omega$
is contractible, so $M$ is aspherical. For a general domain, the Cheng--Yau
construction supplies a canonical Hessian metric on its cone
\cite{ChengYau}. We use this metric to prove a vanishing result for flat
cohomology with coefficients in exterior powers of the holonomy
representation. For oriented four-manifolds, this yields the following:

\begin{maintheorem}\label{thm:four}
Let $M$ be a closed properly convex real projective $4$-manifold.
Then $\chi(M)\geq0$. More precisely, if $M$ is oriented, let $E\to M$ be the flat
rank-five bundle defined by the cone-preserving unimodular lift
of its holonomy. Then
\begin{equation}\label{eq:chi-main}
 10\chi(M)=\dim_{\R}H^2(M;\Lambda^2E^*).
\end{equation}
\end{maintheorem}

The holonomy lift and its flat bundle are constructed in
Section \ref{sec:cone}. No strict convexity, boundary regularity,
or irreducibility hypothesis is imposed, and the proof makes
no assumption on the sectional curvature of a metric on $M$.
The identity \eqref{eq:chi-main} explains the sign by realizing
a positive multiple of the Euler characteristic as the dimension
of a cohomology group.

The result also completes the determination of the possible
Euler characteristics in the oriented four-dimensional case.
Kobayashi, using an earlier Pontryagin-form identity of Avez,
proved the rational vanishing of the Pontryagin classes of real
projective manifolds
\cite[Theorem~4.16]{Kobayashi1984}. In dimension four, this
forces the signature to vanish and the Euler characteristic
to be even. These restrictions alone leave its sign undetermined.
Riolo--Seppi--Slavich construct closed oriented properly convex
projective four-manifolds with every positive even Euler
characteristic \cite[Theorem~1.3]{RioloSeppiSlavich}.
Together with the simplex quotients, for which $\chi=0$,
Theorem \ref{thm:four} shows that the possible values are precisely
the even non-negative integers.

\subsection*{Cohomology with flat coefficients}

The vanishing theorem underlying Theorem \ref{thm:four} holds in
arbitrary dimension. Let $M$ be oriented of dimension $n$,
put $N=n+1$, and let $E$ be the standard flat rank-$N$ bundle
defined by its unimodular holonomy lift. For $0\leq p\leq N$, set
\[
 F_p=\Lambda^pE^*.
\]
The flat connection defines an exterior differential $\dd_p$ on
$\Omega^\bullet(M;F_p)$, and its cohomology is denoted by
$H^\bullet(M;F_p)$.

\begin{maintheorem}\label{thm:vanishing}
Let $M$ be a closed oriented properly convex real projective
$n$-manifold, $n\geq2$, and let $E$ be its standard unimodular
flat bundle. For $0\leq p\leq n+1$ and $0\leq q\leq n$,
\begin{equation}\label{eq:vanishing-intro}
 H^q(M;\Lambda^pE^*)=0
 \quad\text{if}\quad q<\frac{p(n+1-p)}{n+1}.
\end{equation}
\end{maintheorem}

The theorem follows from a Bochner identity on the compact
manifold $M$. The Cheng--Yau metric on the cone determines
a Riemannian metric $k$ on $M$ and a metric
on $F_p$. Let $\dd_p^*$ be the resulting formal adjoint and
let $\Delta_{p,q}=\dd_p\dd_p^*+\dd_p^*\dd_p$. We prove
the quantitative estimate
\begin{equation}\label{eq:gap-intro}
 \|\dd_pa\|^2+\|\dd_p^*a\|^2
 \geq\left(\frac{p(N-p)}N-q\right)\|a\|^2,
 \qquad a\in\Omega^q(M;F_p).
\end{equation}
More precisely, we construct first-order operators
$\mathsf A_{p,q},\mathsf B_{p,q}$ on bundles over $M$ such
that
\begin{equation}\label{eq:Bochner-intro}
 \Delta_{p,q}
 =\mathsf A_{p,q}^*\mathsf A_{p,q}
  +\mathsf B_{p,q}^*\mathsf B_{p,q}
  +\left(p-q-\frac{p^2}{N}\right)\id.
\end{equation}
The bundles, metrics, and operators are defined in
Section \ref{sec:compact-Bochner}. Taking the inner product with $a$
gives \eqref{eq:gap-intro}. The flat connection need not
preserve the fiber metric, but the corresponding elliptic
Hodge theory still represents flat cohomology by harmonic
forms; we include the standard argument in Lemma \ref{lem:flat-Hodge}. \\

Theorem \ref{thm:four} is a consequence of \eqref{eq:vanishing-intro}. For $N=5$, and $p=2,3$ we have
\[
 \frac{p(5-p)}5=\frac65>1.
\]
Theorem~\ref{thm:vanishing} therefore eliminates degrees zero
and one cohomology with coefficients in $F_2$ and $F_3$. A parallel volume form on $E$
identifies $F_2^*$ with $F_3$. Poincar\'e duality then
eliminates degrees three and four for $F_2$. Its cohomology
is thus concentrated in degree two, and the Euler characteristic
with rank-ten flat coefficients yields \eqref{eq:chi-main}.
In higher even dimensions, the vanishing range and its dual
can leave other degrees besides the middle one, so the theorem
does not by itself establish the sign conjecture there.

\subsection*{Organization of the paper}

Section~\ref{sec:cone} constructs the unimodular
holonomy bundle and the radial product.
Section~\ref{sec:double} defines the double-form
operators and derives their identity in the cone over $\Omega$ following \cite{Akagawa}.
Section~\ref{sec:radial} transfers it to the aforementioned
Bochner identity on $M$ and proves Theorem \ref{thm:vanishing}.
Section~\ref{sec:euler} gives the topological argument
for Theorem \ref{thm:four}.
Appendix~\ref{app:Hessian} gives a direct proof in coordinates of the identities in \cite{Akagawa}.

\subsection*{AI declaration} Chat GPT Sol 6.1 was used for literature review, proofreading and editing. All mathematical ideas are original of the author.

\subsection*{Acknowledgment} The author acknowledges support from the Italian Ministry of University and Research (MUR) through the FIS3 project “Geometry and Topology of higher Teichm\"uller spaces”, grant no. FIS-2024-05142.

\section{The Cheng--Yau cone and the holonomy bundle}\label{sec:cone}

We work in dimension $n\geq2$ and write $N=n+1$. Unless explicitly
stated otherwise, $M=\Omega/\Gamma$ is oriented. We begin by fixing the lift of the holonomy.

\subsection{The unimodular lift}

Choose an affine chart in which $\Omega$ is bounded and convex.
After identifying this chart with $\{x^0=1\}\subset\R^N$, let
\[
 \C=\{t(1,z):t>0,\ z\in\Omega\}\subset V=\R^N.
\]
This is an open convex cone. Its radial
projection identifies $\C/\R_{>0}$ with $\Omega$.

\begin{lemma}\label{lem:lift}
The holonomy of an oriented properly convex projective manifold has
a unique lift $\Gamma\to\SL(V)$ that preserves $\C$.
\end{lemma}

\begin{proof}
By \cite[Section~2.1]{BallasDancigerLee}, the holonomy has a
unique lift to
\[
\mathrm{SL}^{\pm}(V)=\{A\in\mathrm{GL}(V):|\det A|=1\}
\]
whose image preserves the cone $\mathcal C$.
We check that this lift takes values in $\mathrm{SL}(V)$.
The radial projection $\mathcal C\to\Omega$ has naturally
oriented fibers: each fiber is a positive ray, oriented by
increasing radial parameter. Every cone-preserving linear
map preserves these fiber orientations. Moreover, the
holonomy acts on $\Omega$ by orientation-preserving maps,
because $M=\Omega/\Gamma$ is oriented. The lifted maps
therefore preserve the induced orientation of $\mathcal C$,
so their determinants are positive. 
Uniqueness follows from the uniqueness of the cone-preserving
lift to $\mathrm{SL}^{\pm}(V)$.
\end{proof}

We use the same symbol $\Gamma$ for the lifted subgroup of
$\SL(V)$. Define
\begin{equation}\label{eq:bundle-E}
 E=(\Omega\times V)/\Gamma,\qquad
 \gamma\cdot(z,v)=(\gamma z,\gamma v).
\end{equation}
Local constant sections on $\Omega$ define its flat connection.
The induced bundle $F_p=\Lambda^pE^*$ has monodromy
$\Lambda^p(\gamma^{-\mathsf T})$. Since $\Gamma\subset\SL(V)$,
any fixed volume form on $V$ defines a parallel volume form on $E$.

\subsection{The canonical potential}

Let $D$ denote the standard flat affine connection on $\C$, and
choose a constant volume form on $V$. We use the following
normalization of the Cheng--Yau canonical potential.

\begin{theorem}[Cheng--Yau]\label{thm:CY}
There is a unique smooth strictly convex function $\phi$ on $\C$
satisfying
\begin{equation}\label{eq:CY}
 \det D^2\phi=e^{2\phi}
\end{equation}
such that the Hessian metric $g=D\dd\phi$ is complete.
\end{theorem}

This is the canonical potential theorem of Cheng--Yau
\cite{ChengYau}; see also \cite[Theorem~1.3 and Corollary~1.1]{Fox}
for a precise formulation and the uniqueness statement.
Completeness is used here to characterize the potential;
the compact Bochner estimate below uses only local identities.
The determinant in
\eqref{eq:CY} is computed in affine coordinates of unit volume for
the chosen constant volume form.

\begin{lemma}\label{lem:potential}
If $A\in\GL(V)$ preserves $\C$, then
\begin{equation}\label{eq:automorphism}
 \phi(Ax)=\phi(x)-\log|\det A|.
\end{equation}
In particular,
\begin{equation}\label{eq:homogeneous}
 \phi(tx)=\phi(x)-N\log t\quad(t>0),\qquad
 \phi(\gamma x)=\phi(x)\quad(\gamma\in\Gamma).
\end{equation}
The transformations $A$ and the radial dilations preserve $g$.
\end{lemma}

\begin{proof}
Set $\psi(x)=\phi(Ax)+\log|\det A|$. The chain rule gives
\[
 D^2\psi(x)=A^{\mathsf T}(D^2\phi)(Ax)A.
\]
It follows that
\[
 \det D^2\psi(x)
 =|\det A|^2 e^{2\phi(Ax)}=e^{2\psi(x)}.
\]
The function $\psi$ is strictly convex, and its Hessian metric
is $A^*g$, which is complete because $A$ is a diffeomorphism of
$\C$. Uniqueness in Theorem \ref{thm:CY} gives $\psi=\phi$,
which proves \eqref{eq:automorphism}. Taking $A=t\id$ gives the
first identity in \eqref{eq:homogeneous}, and taking
$A=\gamma\in\Gamma$ gives the second. Applying $D\dd$ to
\eqref{eq:automorphism} proves $A^*g=g$.
\end{proof}

\subsection{The radial product}

Let $H$ be the vector field, defined by $H_x=x$ in the affine
identification $T_x\C=V$. Set
\[
 \Sigma=\phi^{-1}(0),\qquad k=g|_{T\Sigma}.
\]
Every radial ray meets $\Sigma$ once. Since $\phi$ is
$\Gamma$-invariant, the radial projection induces a diffeomorphism
$\Sigma/\Gamma\cong M$. We use $k$ also for the quotient metric.

\begin{proposition}\label{prop:product}
The map
\[
 \Phi:\R\times\Sigma\longrightarrow\C,
 \qquad \Phi(r,x)=e^rx,
\]
is a $\Gamma$-equivariant diffeomorphism, and
\begin{equation}\label{eq:product}
 \Phi^*g=N\,\dd r^2+k.
\end{equation}
In particular, $X=\C/\Gamma$ identifies with $\R\times M$ with
this product metric. In these coordinates,
\begin{equation}\label{eq:radial-data}
 H=\partial_r,\quad \phi=-Nr,\quad
 |\dd r|_g^2=N^{-1},\quad
 \dd\vol_g=\sqrt N\,\dd r\,\dd\vol_k.
\end{equation}
\end{proposition}

\begin{proof}
In affine coordinates $H=\Sigma_{i} x^i\partial_i$. Differentiating
$\phi(tx)=\phi(x)-N\log t$ in $t$ at $t=1$ gives
\[
 \sum_i x^i\phi_i=-N.
\]
Differentiating this equation in $x^j$ then gives
\[
 \phi_j+\sum_i x^i\phi_{ij}=0.
\]
Since $g_{ij}=\phi_{ij}$, this is the coordinate form of
\begin{equation}\label{eq:Euler-gradient}
 g(H,Y)=-\dd\phi(Y),\qquad \grad_g\phi=-H.
\end{equation}
In particular, $g(H,H)=-\dd\phi(H)=N$, and $H$ is orthogonal
to every level set of $\phi$.

\noindent The inverse of $\Phi$ is
\[
 y\longmapsto
 \left(-\frac{\phi(y)}N,\ e^{\phi(y)/N}y\right).
\]
This proves that $\Phi$ is a diffeomorphism.
To compute its metric in coordinates, write a local
parametrization of $\Sigma$ as $\sigma(z)$ and put
$\sigma_a=\partial_a\sigma$. Then
\[
 d\Phi(\partial_r)=e^r\sigma,\qquad
 d\Phi(\partial_a)=e^r\sigma_a.
\]
Radial dilation preserves $g$ by Lemma \ref{lem:potential}.
Equivalently, for fixed ambient vectors $v,w$,
$g_{e^rx}(v,w)=e^{-2r}g_x(v,w)$.
Therefore the product metric coefficients are
\[
 \begin{split}
 g_{e^r\sigma}(e^r\sigma,e^r\sigma)&=N,\\
 g_{e^r\sigma}(e^r\sigma,e^r\sigma_a)&=0,\\
 g_{e^r\sigma}(e^r\sigma_a,e^r\sigma_b)&=k_{ab}.
 \end{split}
\]
These identities prove \eqref{eq:product}.
Moreover, $r=-\phi/N$, and the displayed inverse shows that
$\Gamma$ acts trivially on $r$ and by its original action on
$\Sigma$. Thus the product is equivariant and descends to
$X=\R\times M$. Finally, the inverse metric and determinant
of the block matrix $\operatorname{diag}(N,k_{ab})$ give
$|\dd r|^2=N^{-1}$ and
$\dd\vol_g=\sqrt N\,\dd r\,\dd\vol_k$, while
$d\Phi(\partial_r)=H$ gives the remaining radial formulas.
\end{proof}

\section{Double forms and Akagawa's identity}\label{sec:double}

The cone has two structures that will be used separately: its flat
affine connection and its Hessian metric. The flat connection defines
the two differentials below, while the metric defines their adjoints.

\subsection{Double forms, exterior operators, and inner products}

Let $(Y,D)$ be a flat affine manifold of dimension $N$. A double form
of bidegree $(p,q)$ is a smooth section of
\[
 \Lambda^pT^*Y\otimes\Lambda^qT^*Y.
\]
We denote the space of these sections by $\mathcal A^{p,q}(Y)$.
In affine coordinates $x^1,\ldots,x^N$, write
\[
 \omega=\frac1{p!\,q!}
 \omega_{i_1\ldots i_p\,;\,j_1\ldots j_q}
 \dd x^{i_1}\wedge\cdots\wedge\dd x^{i_p}
 \otimes
 \dd x^{j_1}\wedge\cdots\wedge\dd x^{j_q}.
\]
Here and below repeated indices are summed from $1$ to $N$ unless
otherwise indicated. The components are alternating separately in
the indices on the two sides of the semicolon; no symmetry between
the two groups is imposed.

Exterior multiplication takes place independently in the two
factors. In particular,
\[
 (\xi\otimes\eta)\wedge(\xi'\otimes\eta')
 =(\xi\wedge\xi')\otimes(\eta\wedge\eta')
\]
for decomposable double forms. There is no sign from passing a form
in the first factor past a form in the second factor. This is the
convention of \cite[Definition 2.1.2]{Akagawa}. For a covector $\xi$,
let $\varepsilon(\xi)$ and $\bar\varepsilon(\xi)$ denote wedge
multiplication in the first and second factors. For a vector $Z$,
let $\iota(Z)$ and $\bar\iota(Z)$ denote contraction in the
respective factors. Our contraction convention is
\[
 (\iota(Z)\omega)_{i_2\ldots i_p;J}
 =Z^a\omega_{a i_2\ldots i_p;J},\qquad
 (\bar\iota(Z)\omega)_{I;j_2\ldots j_q}
 =Z^a\omega_{I;a j_2\ldots j_q}.
\]
Thus contraction inserts the vector in the first slot of the
chosen factor.

Write $e_i=\varepsilon(\dd x^i)$,
$\bar e_i=\bar\varepsilon(\dd x^i)$,
$i_i=\iota(\partial_i)$, and
$\bar i_i=\bar\iota(\partial_i)$. In each factor the usual
exterior-algebra relations hold:
\[
 i_i e_j+e_j i_i=\delta_{ij}\id,\qquad
 \bar i_i\bar e_j+\bar e_j\bar i_i=\delta_{ij}\id.
\]
Two wedge operators, or two contraction operators, anticommute
within the same factor. All operators acting on different factors
commute.

Put $D_i=D_{\partial_i}$. In an affine coordinate frame, $D_i$
differentiates the components of the entire double form ordinarily.
Define
\begin{equation}\label{eq:double-differentials}
 \partial=\sum_i e_iD_i,\qquad
 \bdel=\sum_i\bar e_iD_i.
\end{equation}
For example,
\[
 (\partial\omega)_{i_0\ldots i_p;J}
 =\sum_{r=0}^p(-1)^r
   \partial_{i_r}\omega_{i_0\ldots\widehat{i_r}\ldots i_p;J},
\]
and $\bdel$ has the identical formula in the second group of
indices. These definitions are independent of the affine chart.
Since affine derivatives commute and the two exterior factors
commute, they give
\begin{equation}\label{eq:double-relations}
 \partial^2=\bdel^2=0,\qquad
 \partial\bdel=\bdel\partial.
\end{equation}
The last equality has no minus sign. In particular, the two factors
are not being combined into a single exterior algebra with the
usual total-degree sign convention.

For fixed $p$, the first exterior factor is a flat coefficient
bundle. More precisely,
\begin{equation}\label{eq:twisted-interpretation}
 (\mathcal A^{p,\bullet}(Y),\bdel)
 =\bigl(\Omega^\bullet(Y;\Lambda^p(T^*Y,D)),
       \dd_{\Lambda^p(T^*Y,D)}\bigr).
\end{equation}
To check this description, choose any local parallel frame $s_A$
of the first exterior factor and write
$\omega=\sum_A s_A\otimes\omega^A$. Then
\[
 \bdel\omega=\sum_A s_A\otimes\dd\omega^A.
\]
The differential on the right requires only the connection on the
coefficient bundle; the $\omega^A$ are ordinary differential forms.

Suppose now that $g=D\dd\phi$ is positive definite, and write
$g_{ij}=\phi_{ij}$ and $(g^{ij})=(g_{ij})^{-1}$. The metric on
$T^*Y$ is $g^{-1}$. We use its determinant metric on exterior
powers and the tensor-product metric on double forms. Explicitly,
\[
 \langle\omega,\eta\rangle_g
 =\frac1{p!\,q!}
 \omega_{I;J}\eta_{K;L}
 \left(\prod_{a=1}^p g^{i_a k_a}\right)
 \left(\prod_{b=1}^q g^{j_b l_b}\right).
\]
For example, on a double $(1,1)$-form this is
$g^{ik}g^{j\ell}\omega_{i;j}\eta_{k;\ell}$.
The $L^2$ inner product is defined with the Riemannian density:
\[
 (\omega,\eta)_Y
 =\int_Y\langle\omega,\eta\rangle_g\,
        \sqrt{\det(g_{ij})}\,\dd x^1\cdots\dd x^N.
\]
The notation $\mathcal A_c^{p,q}(Y)$ means compact support.

Let $\delta$ and $\bdelt$ be the formal adjoints of $\partial$
and $\bdel$ for these inner products. They are determined by
integration by parts on compactly supported smooth test forms.
Define the local differential operators
\[
 \Box=\partial\delta+\delta\partial,\qquad
 \bBox=\bdel\bdelt+\bdelt\bdel.
\]

\subsection{The Hessian identities}

Regard $g$ as the double $(1,1)$-form
$g_{ij}\,\dd x^i\otimes\dd x^j$, and let $L$ be wedge
multiplication by this double form. Its pointwise adjoint is
\[
 \Lambda=g^{ij}i_i\bar i_j.
\]
Indeed, $e_i^*=g^{ia}i_a$ and
$\bar e_j^*=g^{jb}\bar i_b$, so
$L^*=g_{ij}g^{ia}g^{jb}i_a\bar i_b=g^{ab}i_a\bar i_b$.
In components,
\[
 (\Lambda\omega)_{I;J}=g^{ab}\omega_{aI;bJ}.
\]
Thus $\Lambda$ traces one covariant index from each factor. For a
$(1,1)$-form, $\Lambda\omega=g^{ij}\omega_{i;j}$.

The first Koszul form and its metric dual are
\[
 \alpha=\tfrac12\dd\log\det(g_{ij}),\qquad
 X_\alpha=\alpha^{\sharp_g}.
\]
An affine coordinate change multiplies $\det(g_{ij})$ by a positive
constant, so its logarithmic differential is well defined. The
local Hessian K\"ahler identities are
\begin{equation}\label{eq:Kahler}
 \Lambda\partial+\partial\Lambda
   =-\bdelt+\bar\iota(X_\alpha),\qquad
 \Lambda\bdel+\bdel\Lambda
   =-\delta+\iota(X_\alpha).
\end{equation}
These are \cite[Theorem 2.3.6]{Akagawa}. In Appendix
\ref{app:Hessian}, we prove \eqref{eq:Kahler} directly in coordinates.

For the Cheng--Yau potential, its Monge--Amp\`ere equation gives
$\frac12\dd\log\det(D^2\phi)=\dd\phi$. The radial identity
\eqref{eq:Euler-gradient} then gives
\begin{equation}\label{eq:alpha}
 \alpha=\dd\phi,\qquad X_\alpha=-H.
\end{equation}
The following calculation explains how the Euler field introduces
the exterior degrees.

\begin{lemma}\label{lem:Cartan}
On double forms of bidegree $(p,q)$ on $\C$,
\begin{equation}\label{eq:Cartan}
 \partial\iota(H)+\iota(H)\partial=D_H+p\id,
 \qquad
 \bdel\bar\iota(H)+\bar\iota(H)\bdel=D_H+q\id.
\end{equation}
\end{lemma}

\begin{proof}
Since $D_iH=\partial_i$, differentiation of contraction gives
\[
 D_i(\iota(H)\omega)=i_i\omega+\iota(H)D_i\omega.
\]
Therefore
\begin{align*}
 (\partial\iota(H)+\iota(H)\partial)\omega
 &=\sum_i e_i i_i\omega
   +\sum_i(e_i\iota(H)+\iota(H)e_i)D_i\omega\\
 &=\sum_i e_i i_i\omega+D_H\omega.
\end{align*}
The second equality uses
$e_i\iota(H)+\iota(H)e_i=x^i\id$.
On an elementary exterior monomial of first degree $p$,
the operator $\sum_i e_i i_i$ counts its $p$ indices:
a term vanishes unless $\dd x^i$ occurs, and contraction
followed by reinsertion gives that monomial for each such
index. Thus $\sum_i e_i i_i=p\id$.
The identical calculation in the second factor gives
$\sum_i\bar e_i\bar i_i=q\id$ and proves the second identity.
These formulas also appear in
\cite[Corollary~3.3.3]{Akagawa}.
\end{proof}

The Lie derivative in the $H$-direction of the entire covariant tensor is
\[
 (\mathcal L_H\omega)_{I;J}
 =H^a\partial_a\omega_{I;J}+(p+q)\omega_{I;J},
\]
because $\partial_iH^a=\delta_i^a$ contributes once for every
covariant slot. Hence
$\mathcal L_H\omega=D_H\omega+(p+q)\omega$.
In particular, dilation invariance means
$D_H\omega=-(p+q)\omega$.

\begin{proposition}[Akagawa's cone identity]\label{prop:Akagawa}
On double forms of bidegree $(p,q)$ on the Cheng--Yau cone,
\begin{equation}\label{eq:Akagawa}
 \bBox=\Box+(p-q)\id.
\end{equation}
\end{proposition}

\begin{proof}
This is the local identity established in
\cite[proof of Main Theorem 3, p.~86]{Akagawa}. Put
$I=\iota(X_\alpha)$ and $\bar I=\bar\iota(X_\alpha)$.
Solving \eqref{eq:Kahler} for the adjoints gives
\[
 \delta=-\Lambda\bdel-\bdel\Lambda+I,
 \qquad
 \bdelt=-\Lambda\partial-\partial\Lambda+\bar I.
\]
Substitution into the two Laplacians yields
\begin{align*}
 \Box={}&-\partial\Lambda\bdel-\partial\bdel\Lambda
          -\Lambda\bdel\partial-\bdel\Lambda\partial
          +\partial I+I\partial,\\
 \bBox={}&-\bdel\Lambda\partial-\bdel\partial\Lambda
          -\Lambda\partial\bdel-\partial\Lambda\bdel
          +\bdel\bar I+\bar I\bdel.
\end{align*}
The four terms involving $\Lambda$ agree between the two lines:
the only interchanges needed are
$\bdel\partial=\partial\bdel$, with $\Lambda$ left in its
original position. Their difference is therefore
\[
 \bBox-\Box
 =-(\partial I+I\partial)
   +(\bdel\bar I+\bar I\bdel).
\]
By \eqref{eq:alpha}, $I=-\iota(H)$ and
$\bar I=-\bar\iota(H)$. Applying Lemma \ref{lem:Cartan} gives
\[
 \bBox-\Box=(D_H+p)-(D_H+q)=(p-q)\id.
\]
This proves the asserted identity.
\end{proof}

\begin{proposition}\label{prop:quotient-identity}
The four operators $\partial,\bdel,\delta,\bdelt$ descend to
$X=\C/\Gamma$, and \eqref{eq:Akagawa} holds on $X$.
\end{proposition}

\begin{proof}
Each element of $\Gamma$ is affine for $D$ and preserves $g$.
Its pullback therefore commutes with both differentials and
preserves the pointwise tensor metrics and the Riemannian density.
It also commutes with their formal adjoints: the defining
integration-by-parts identity on a coordinate neighborhood is
unchanged by this pullback. 
\end{proof}

The next section rewrites \eqref{eq:Akagawa} as an identity
of operators on $M$.

\section{A Bochner identity}\label{sec:radial}

The aim of this section is to express the two sides of Akagawa's
cone identity as operators on bundles over $M$. In particular, we define the
operators in the Bochner identity \eqref{eq:Bochner-intro} and compute the action
of $\bBox$ and $\Box$ on dilation invariant forms.

\subsection{The coefficient bundle and its metric}\label{sec:bundle}

The identification $T\C=\C\times V$ identifies its flat
cotangent bundle with $\C\times V^*$. The action on the fiber is
$\gamma^{-\mathsf T}$. Its $p$th exterior power therefore
descends to a flat bundle $\mathcal F_p$ on $X=\C/\Gamma$.
Radial projection gives a flat bundle isomorphism
\begin{equation}\label{eq:coefficient-identification}
 \mathcal F_p\cong\pi_M^*F_p,
 \qquad \pi_M:\R\times M\longrightarrow M.
\end{equation}
To describe it, write a point of the cone as $e^r\sigma(y)$,
where $\sigma(y)\in\Sigma$. A covector with affine components
$\lambda_i$ at $\sigma(y)$ is identified with the covector
having the same components $\lambda_i$ at $e^r\sigma(y)$.
This is flat parallel transport in the ambient space $V$.
It is $\Gamma$-equivariant, so the description gives the stated
isomorphism on the quotient. In particular, the connection in the radial direction is ordinary differentiation.

Under \eqref{eq:coefficient-identification}, $\bdel$ is the
exterior covariant derivative of $\pi_M^*F_p$, as explained in
\eqref{eq:twisted-interpretation}. In a local flat coefficient
frame $s_A$, it is characterized by
\[
 \bdel\Bigl(\sum_A s_A\otimes\omega^A\Bigr)
       =\sum_A s_A\otimes\dd\omega^A,
\]
where $\omega^A$ are ordinary differential forms on $\R\times M$.

The inverse Cheng--Yau metric \(g^{-1}\) induces a fiber metric on
\(\mathcal F_p=\Lambda^pT^*X\). Under the identification \eqref{eq:coefficient-identification}, the fiber of \(\mathcal F_p\) at \((0,y)\) is identified with \((F_p)_y\).
We denote by \(H_{p,y}\) the inner product transported to \((F_p)_y\)
through this identification. This defines a fiber metric \(H_p\)
on \(F_p\).  

\begin{lemma}\label{lem:scaling}
Under \eqref{eq:coefficient-identification}, the fiber metric on
$\mathcal F_p$ over $(r,y)$ is
\begin{equation}\label{eq:fiber-scaling}
 e^{2pr}H_p(y).
\end{equation}
\end{lemma}

\begin{proof}
Fix $t>0$ and differentiate
$\phi(tx)=\phi(x)-N\log t$ in the affine variables $x$.
The first and second derivatives give
\[
 t\phi_i(tx)=\phi_i(x),\qquad
 t^2\phi_{ij}(tx)=\phi_{ij}(x).
\]
Consequently, for the matrices of the metric and its inverse,
\begin{equation}\label{eq:metric-scaling-components}
 g_{ij}(e^rx)=e^{-2r}g_{ij}(x),\qquad
 g^{ij}(e^rx)=e^{2r}g^{ij}(x).
\end{equation}
For covectors $\lambda,\mu$ we then have
\[
 \langle\lambda,\mu\rangle_{e^rx}
 =\sum_{i,j}g^{ij}(e^rx)\lambda_i\mu_j
 =e^{2r}\langle\lambda,\mu\rangle_x.
\]
The exterior metric is the determinant metric:
\[
 \left\langle
 \lambda_1\wedge\cdots\wedge\lambda_p,\,
 \mu_1\wedge\cdots\wedge\mu_p
 \right\rangle
 =\det\bigl(\langle\lambda_a,\mu_b\rangle\bigr)_{a,b=1}^p.
\]
Each entry of this $p\times p$ matrix acquires the factor
$e^{2r}$, and its determinant acquires $e^{2pr}$.
Decomposable exterior tensors span the fiber, so bilinearity
proves \eqref{eq:fiber-scaling}.
\end{proof}

We regard double \((p,q)\)-forms on \(X\) as
ordinary \(q\)-forms on \(\mathbb R\times M\) with coefficients
in \(\pi_M^*F_p\):
\[
\Lambda^pT^*X\otimes\Lambda^qT^*X
\cong
\pi_M^*F_p\otimes\Lambda^qT^*(\mathbb R\times M).
\]
Their pointwise metric is therefore the tensor product of the metric in Lemma \ref{lem:scaling} with the metric on ordinary \(q\)-forms induced by the product Riemannian metric \(N\,dr^2+k\).
In particular, for \(s\in(F_p)_y\) and
\(\alpha\in\Lambda^qT^*_{(r,y)}(\mathbb R\times M)\),
\[
|s\otimes\alpha|_g^2
=
e^{2pr}|s|_{H_p}^2\,|\alpha|_{N\,dr^2+k}^2.
\]

\noindent On \(\Omega^q(M;F_p)\), we use the pointwise metric
\(H_p\otimes\Lambda^q k^{-1}\) and the volume density
\(d\mathrm{vol}_k\). Let $\dd_p$ be the flat exterior derivative,
let $\dd_p^*$ be its formal adjoint, and put
\begin{equation}\label{eq:Q}
 \cQ_{p,q}(a)=\|\dd_pa\|_M^2+\|\dd_p^*a\|_M^2,\qquad
 \Delta_{p,q}=\dd_p\dd_p^*+\dd_p^*\dd_p.
\end{equation}
Thus $\cQ_{p,q}(a)=(\Delta_{p,q}a,a)_M$.

\subsection{The differential and its formal adjoint on the product}\label{sec:product}

Every $\mathcal F_p$-valued $q$-form on $X$ has a unique
decomposition
\[
 \omega=a(r)+\dd r\wedge b(r),
\]
where $a(r)\in\Omega^q(M;F_p)$ and
$b(r)\in\Omega^{q-1}(M;F_p)$. The form is horizontal if $b=0$.
Wedge multiplication by $\dd r$ in this decomposition acts in
the second exterior factor. By definition of the metric (see \eqref{eq:radial-data}), together with
$\dd\vol_g=\sqrt N\,\dd r\,\dd\vol_k$, give
\begin{equation}\label{eq:product-inner}
 (a+\dd r\wedge b,c+\dd r\wedge d)_X
 =\sqrt N\int_\R e^{2pr}
   \left((a,c)_M+\frac1N(b,d)_M\right)\dd r
\end{equation}
for compactly supported families.

\begin{lemma}\label{lem:product-operators}
In this decomposition,
\begin{align}
 \bdel(a+\dd r\wedge b)
 &=\dd_pa+\dd r\wedge(\partial_ra-\dd_pb),
 \label{eq:product-d}\\
 \bdelt(a+\dd r\wedge b)
 &=\dd_p^*a-\frac1N(\partial_r+2p)b
                     -\dd r\wedge\dd_p^*b.
 \label{eq:product-adjoint}
\end{align}
\end{lemma}

\begin{proof}
In a flat coefficient frame, $\bdel$ is ordinary exterior
differentiation of the second-factor forms. Thus, 
\[
 \bdel a=\dd_pa+\dd r\wedge\partial_ra,\qquad
 \bdel(\dd r\wedge b)=-\dd r\wedge\dd_pb,
\]
which proves \eqref{eq:product-d}.

To compute the adjoint, take a compactly supported test form
$c(r)+\dd r\wedge d(r)$ with $c(r)\in \Omega^{q-1}(M,F_{p})$ and $d(r)\in \Omega^{q-2}(M,F_{p})$. Its derivative paired
with $\omega=a+\dd r\wedge b$ is
\[
 \sqrt N\int_\R e^{2pr}
 \left((a,\dd_pc)_M+
       \frac1N(b,\partial_rc-\dd_pd)_M\right)\dd r.
\]
The terms involving $\dd_p$ become
$(\dd_p^*a,c)_M$ and $(-\dd_p^*b,d)_M$ by integration
on $M$. Compact support in $r$ gives
\[
 \int_\R e^{2pr}(b,\partial_rc)_M\,\dd r
 =-\int_\R e^{2pr}(\partial_rb+2pb,c)_M\,\dd r.
\]
The pairing is therefore
\[
 \sqrt N\int_\R e^{2pr}
 \left(
 \left(\dd_p^*a-\frac1N(\partial_r+2p)b,c\right)_M
 +\frac1N(-\dd_p^*b,d)_M
 \right)\dd r.
\]
Comparison with \eqref{eq:product-inner} proves
\eqref{eq:product-adjoint}.
\end{proof}

For a horizontal form $f(r)a(y)$, the formulas reduce to
\begin{equation}\label{eq:horizontal-operators}
 \bdel(fa)=f\,\dd_pa+f'\dd r\wedge a,\qquad
 \bdelt(fa)=f\,\dd_p^*a.
\end{equation}
We will need the Laplacian on such forms. Applying
\eqref{eq:product-adjoint} to the two components of
$\bdel(fa)$ gives
\begin{align*}
 \bdelt\bdel(fa)
 &=f\,\dd_p^*\dd_pa
   -\frac1N(f''+2pf')a-f'\dd r\wedge\dd_p^*a,\\
 \bdel\bdelt(fa)
 &=f\,\dd_p\dd_p^*a+f'\dd r\wedge\dd_p^*a.
\end{align*}
The radial exterior terms cancel, so
\begin{equation}\label{eq:horizontal-barBox}
 \bBox(fa)=f\,\Delta_{p,q}a-\frac1N(f''+2pf')a.
\end{equation}
This holds for arbitrary smooth functions $f$.

\subsection{A Bochner identity on the compact quotient}\label{sec:compact-Bochner}

We now define all the bundles and operators needed for the Bochner identity \eqref{eq:Bochner-intro}. The level hypersurface $\Sigma$ gives a section
$\sigma$ of $E$: at the point represented by $x\in\Sigma$,
its value is $x\in V$. The equivariance of $\Sigma$ makes this
section well defined on $M$. Let $D^E$ be the flat connection
on $E$. Since $D^E_Y\sigma$ is the derivative of the immersion
of $\Sigma$, it identifies $TM$ with a subbundle of $E$.
Moreover,
$\dd\phi(\sigma)=-N$ and $\dd\phi(D^E_Y\sigma)=0$, so
\[
 E=\R\sigma\oplus D^E\sigma(TM).
\]
In this subsection, a tangent vector $Y$ regarded as an element
of $E$ means $D^E_Y\sigma$.

\noindent The cone metric on the slice defines a fiber metric $G$ on $E$:
\begin{equation}\label{eq:E-splitting}
 G(\sigma,\sigma)=N,\qquad G(\sigma,Y)=0,\qquad
 G(Y,Z)=k(Y,Z).
\end{equation}
The metric on \(\Lambda^pE^*\) induced by \(G\) is precisely
the metric \(H_p\) defined in Section \ref{sec:bundle}.
Define $\zeta\in E^*$ by $\zeta(\sigma)=1$ and
$\zeta(TM)=0$. Then
\[
 \zeta=G(\sigma,\cdot)/N,\qquad
 \zeta^{\sharp_G}=\sigma/N,\qquad |\zeta|^2=N^{-1}.
\]
For any local frame $e_a$ of $TM$, let $\theta^a\in E^*$
annihilate $\sigma$ and be dual to $e_a$ on $TM$. When working on the cone, we use the same symbols
\(\zeta,\theta^a\) for their extensions by flat parallel
transport along the rays, so their affine components are
kept fixed. In a coordinate frame \(e_a=\partial_{y^a}\),
write \(\sigma_a=\partial_{y^a}\sigma\). The product
parametrization \(\Phi(r,y)=e^r\sigma(y)\) satisfies
\[
d\Phi=e^r\sigma\,dr+e^r\sum_b\sigma_b\,dy^b.
\]
Applying the defining duality relations gives
\[
\begin{aligned}
\Phi^*\zeta
&=e^r\left(\zeta(\sigma)\,dr
          +\sum_b\zeta(\sigma_b)\,dy^b\right)
 =e^r\,dr,\\
\Phi^*\theta^a
&=e^r\left(\theta^a(\sigma)\,dr
          +\sum_b\theta^a(\sigma_b)\,dy^b\right)
 =e^r\,dy^a.
\end{aligned}
\]
Thus, suppressing pullback by \(\Phi\), we write
\[
\zeta=e^r\,dr,
\qquad
\theta^a=e^r\,dy^a.
\]
In particular,
\(\theta^{a_1}\wedge\cdots\wedge\theta^{a_q}
=e^{qr}dy^{a_1}\wedge\cdots\wedge dy^{a_q}\).\\

\noindent For $0\leq p,q\leq N$, set
\[
 \mathcal B^{p,q}=\Lambda^pE^*\otimes\Lambda^qE^*,
\]
with the tensor metric induced by $G$ and the base measure
$\dd\vol_k$. The inclusion of covectors annihilating $\sigma$
gives an isometric bundle map
\begin{equation}\label{eq:horizontal-inclusion}
 J_q:\Lambda^pE^*\otimes\Lambda^qT^*M
                  \longrightarrow\mathcal B^{p,q}.
\end{equation}
It is the identity on the first factor. In a coordinate coframe,
\[
 J_q(\lambda\otimes\dd y^{a_1}\wedge\cdots\wedge\dd y^{a_q})
 =\lambda\otimes\theta^{a_1}\wedge\cdots\wedge\theta^{a_q}.
\]
Note that
\begin{equation}\label{eq:horizontal-decomposition}
 \Lambda^qE^*
 =J_q(\Lambda^qT^*M)
 \oplus\zeta\wedge J_{q-1}(\Lambda^{q-1}T^*M).
\end{equation}
Here the notation $J_q$ in the single exterior factor denotes
the same inclusion without the first-factor identity.
Thus $J_q^*$ discards the terms containing a second-factor
$\zeta$ and identifies the horizontal part with ordinary forms.
In particular, $J_q^*J_q=\id$.

\noindent Let $D^{\mathrm{tot}}$ be the connection on
$\mathcal B^{p,q}$ induced by $D^E$ on both factors, in other words
\[
 D^{\mathrm{tot}}_X(\lambda\otimes\eta)
 =(D_X\lambda)\otimes\eta+\lambda\otimes D_X\eta,
\]
with the induced exterior-power connections on $\lambda,\eta$.
In a local flat frame of $E$, this is ordinary differentiation
of all components.

\noindent Define
\begin{equation}\label{eq:angular-P}
 \begin{split}
 \mathsf P_{p,q}u
 &=\sum_{a=1}^n\varepsilon(\theta^a)
                         D^{\mathrm{tot}}_{e_a}u
       -(p+q)\varepsilon(\zeta)u,\\
 \mathsf P_{p,q}&:\Gamma(\mathcal B^{p,q})
                         \longrightarrow\Gamma(\mathcal B^{p+1,q}).
 \end{split}
\end{equation}

\noindent Let
$\mathsf P_{p,q}^*:\Gamma(\mathcal B^{p+1,q})
\to\Gamma(\mathcal B^{p,q})$
be its formal adjoint for the metrics induced by $G$ and the
volume $\dd\vol_k$, i.e.
\begin{equation}\label{eq:P-formal-adjoint}
 (\mathsf P_{p,q}u,v)_M=(u,\mathsf P_{p,q}^*v)_M
\end{equation}
for all smooth sections $u,v$. \\

\noindent For the adjoint calculation below, we will also use the following identity

\begin{lemma} The adjoint of exterior multiplication by $\zeta$ is the operator
\[
 \varepsilon(\zeta)^*=\iota(\sigma/N).
\]
\end{lemma}
\begin{proof}
At each point, choose a $G$-orthonormal basis
$v_0=\sigma/\sqrt N,v_1,\ldots,v_n$, with dual basis
$\omega^0,\ldots,\omega^n$. Then $\zeta=\omega^0/\sqrt N$.
On the orthonormal basis of exterior monomials, exterior
multiplication by $\omega^0$ and contraction by $v_0$ have
transposed matrices. Hence
\[
\varepsilon(\zeta)^*
=
\frac1{\sqrt N}\varepsilon(\omega^0)^*
=
\frac1{\sqrt N}\iota(v_0)
=
\iota(\sigma/N).
\]
The identity extends to double forms since both operators act
as the identity on the second exterior factor.
\end{proof}

We set exterior powers and corresponding operators to zero
outside degrees $0,\ldots,N$. Define
\begin{equation}\label{eq:compact-AB}
 \mathsf A_{p,q}=\mathsf P_{p,q}J_q,\qquad
 \mathsf B_{p,q}=\mathsf P_{p-1,q}^*J_q.
\end{equation}
Their domains and targets are
\[
 \begin{split}
 \mathsf A_{p,q}&:\Omega^q(M;F_p)
                         \longrightarrow\Gamma(\mathcal B^{p+1,q}),\\
 \mathsf B_{p,q}&:\Omega^q(M;F_p)
                         \longrightarrow\Gamma(\mathcal B^{p-1,q}).
 \end{split}
\]

\begin{proposition}\label{prop:compact-Bochner}
On $\Omega^q(M;F_p)$ one has the differential-operator identity
\begin{equation}\label{eq:compact-Bochner}
 \Delta_{p,q}
 =\mathsf A_{p,q}^*\mathsf A_{p,q}
  +\mathsf B_{p,q}^*\mathsf B_{p,q}
  +\left(p-q-\frac{p^2}{N}\right)\id.
\end{equation}
Consequently,
\begin{equation}\label{eq:compact-squares}
 \cQ_{p,q}(a)
 =\|\mathsf A_{p,q}a\|_M^2+\|\mathsf B_{p,q}a\|_M^2
   +\left(\frac{p(N-p)}N-q\right)\|a\|_M^2.
\end{equation}
\end{proposition}

\begin{proof}
Put $X=\mathcal C/\Gamma\simeq\mathbb R\times M$, and let
$m_s:X\to X$ be the map induced by the dilation $x\mapsto e^sx$ on $V$.
We first define the invariant extension map.
Recall that 
\[
\mathcal B^{p,q}
=\Lambda^pE^*\otimes\Lambda^qE^*\longrightarrow M.
\]
Define the set of dilation invariant sections as
\[
\mathcal A^{p,q}(X)^{\mathrm{dil}}
=
\left\{
\omega\in
\Gamma^\infty\!\left(
X,\Lambda^pT^*X\otimes\Lambda^qT^*X
\right) \ | \ 
m_s^*\omega=\omega
\text{ for every }s\in\mathbb R
\right\}.
\]
Then, the map 
\[
\mathcal U_{p,q}:
\Gamma^\infty(M,\mathcal B^{p,q})
\longrightarrow
\mathcal A^{p,q}(X)^{\mathrm{dil}}
\]
defined by the componentwise formula
\begin{equation}\label{eq:U-definition}
(\mathcal U_{p,q}u)_{I;J}(e^r\sigma(y))
=e^{-(p+q)r}u_{I;J}(y)
\end{equation}
is an isomorphism. Indeed, pullback by $m_s$ multiplies by $e^{(p+q)s}$, so
\[
\bigl(m_s^*\mathcal U_{p,q}u\bigr)_{I;J}
       (e^r\sigma(y))
=
e^{(p+q)s}e^{-(p+q)(r+s)}u_{I;J}(y)
=
(\mathcal U_{p,q}u)_{I;J}(e^r\sigma(y)).
\]
Conversely, dilation invariance forces precisely this
scaling of the affine components. Its inverse, denoted by $\mathcal R_0$, evaluates the full tensor at $r=0$.

We next compute the operator induced by $\partial$ on dilation invariant double forms. 
Put $m=p+q$. In local coordinates $y^a$ on $M$, we have
\[
D_{\partial_r}(\mathcal U_{p,q}u)
=-m e^{-mr}u,
\qquad
D_{\partial_a}(\mathcal U_{p,q}u)
=e^{-mr}D^{\mathrm{tot}}_{\partial_a}u.
\]
Using
\[
\partial
=\varepsilon(\dd r)D_{\partial_r}
 +\sum_a\varepsilon(\dd y^a)D_{\partial_a}
\]
we compute 
\[ \partial(\mathcal U_{p,q}u)=
e^{-(m+1)r}
\left(
\sum_a\varepsilon(\theta^a)
             D^{\mathrm{tot}}_{\partial_a}u
-m\varepsilon(\zeta)u
\right).
\]
The expression in parentheses is
$\mathsf P_{p,q}u$ by definition. Therefore,
\begin{equation}\label{eq:U-first-differential}
\partial(\mathcal U_{p,q}u)
=\mathcal U_{p+1,q}(\mathsf P_{p,q}u).
\end{equation}

To identify its adjoint operator, choose local frames of
$\mathcal B^{p-1,q}$ and $\mathcal B^{p,q}$ over an open subset of $M$,
and extend their elements invariantly to $X$ using
$\mathcal U_{p-1,q}$ and $\mathcal U_{p,q}$, respectively.
These extensions form local frames of the corresponding double-form
bundles on $X$. In these invariant frames, the operator $\partial$ acts like
\[
\partial
=
P_{p-1,q}+\varepsilon(\zeta)\partial_r.
\]
Indeed, its action on the frame elements is given by
$\partial\mathcal U=\mathcal U P$, while its action on their scalar
coefficients follows from the Leibniz rule. 
The formal adjoint can now be computed in these same frames.
We first make a remark about the metric used to compute the adjoint in this frame.
Under flat affine transport, the metric on covectors scales by
$e^{2r}$. Hence the induced metric on double forms of bidegree
$(p,q)$ scales by $e^{2(p+q)r}$. On the other hand, each invariant
extension has the coefficient factor $e^{-(p+q)r}$. Consequently,
for $u,v\in\Gamma^\infty(M,\mathcal B^{p,q})$,
\[
\begin{aligned}
\big\langle\mathcal U_{p,q}u,\mathcal U_{p,q}v\big\rangle(r,y)
=
e^{2(p+q)r}e^{-2(p+q)r}\langle u,v\rangle_G(y)
=
\langle u,v\rangle(y).
\end{aligned}
\]
Thus the invariant extension is a pointwise isometry, and the
Gram matrices of the invariant frames are independent of $r$.
Moreover,
\[
d\operatorname{vol}_g
=
\sqrt N\,dr\,d\operatorname{vol}_k.
\]
Integration by parts on compactly supported test forms therefore
shows that the radial derivative of the scalar coefficients
in these frames has formal adjoint
\[
\partial_r^*=-\partial_r.
\]
Furthermore,
\[
\varepsilon(\zeta)^*
=
\iota(\sigma/N).
\]
Consequently,
\[
\delta
=
P_{p-1,q}^*
-
\iota(\sigma/N)\partial_r.
\]
Here $\partial_r$ differentiates the scalar coefficients relative to
the invariant frames. The derivatives of the frame elements themselves
are already included in $P_{p-1,q}$. For $u\in\Gamma^\infty(M,\mathcal B^{p,q})$, the coefficients of
$\mathcal U_{p,q}u$ relative to these frames are independent of $r$.
It follows that
\[
\delta(\mathcal U_{p,q}u)
=
\mathcal U_{p-1,q}(P_{p-1,q}^*u).
\]
Together with the corresponding formula for $\partial$, this gives
\begin{equation}\label{box}
\Box(\mathcal U_{p,q}u)
=
\mathcal U_{p,q}(L_{p,q}u),
\end{equation}
where
\[
L_{p,q}
=
P_{p,q}^*P_{p,q}
+
P_{p-1,q}P_{p-1,q}^*.
\]

We next compute \(\bar\Box\) on the extension of \(J_q\eta\), where
\(\eta\in\Omega^q(M;F_p)\). Recall that in a flat
coefficient frame, if
\[
\eta=\sum_{I,A}\eta_{I;A}(y)\,dx^I\otimes dy^A 
\]
then the inclusion \(J_q\) replaces \(dy^A\) by \(\theta^A\) at the slice.
Since \(\theta^A=e^{qr}dy^A\), in flat coordinates we have
\begin{equation}\label{6}
\mathcal U_{p,q}(J_q\eta)=e^{-pr}\eta.
\end{equation}
We already established in Section \ref{sec:product} that
\[
\bar\Box(f(r)\eta)
=
f(r)\Delta_{p,q}\eta-\frac1N(f''+2pf')\eta.
\]
Taking \(f=e^{-pr}\), we have
\[
\bar\Box(f(r)\eta)
=
e^{-pr}\Delta_{p,q}\eta+\frac{p^2}{N}e^{-pr}\eta.
\]
Together with \eqref{6}, this yields
\begin{equation}\label{barbox}
\bar\Box(\mathcal U_{p,q}J_q\eta)
=
\mathcal U_{p,q}J_q
\left(\Delta_{p,q}\eta+\frac{p^2}{N}\eta\right).
\end{equation}
Substituting \eqref{box} and \eqref{barbox} into the cone identity
\[
\bar\Box=\Box+(p-q)\mathrm{Id},
\]
we obtain
\[
\begin{aligned}
\mathcal U_{p,q}
\left[J_q\left(\Delta_{p,q}\eta+\frac{p^2}{N}\eta\right)\right]=
\mathcal U_{p,q}
\left[L_{p,q}J_q\eta+(p-q)J_q\eta\right].
\end{aligned}
\]
Applying $\mathcal{R}_0$, we obtain
\[
J_q\left(\Delta_{p,q}\eta+\frac{p^2}{N}\eta\right)
=
L_{p,q}J_q\eta+(p-q)J_q\eta.
\]
Finally, we apply \(J_q^*\) on both sides and, since \(J_q^*J_q=\mathrm{Id}\) and the adjoints of $A_{p,q}$ $B_{p,q}$ satisfy
\[
A_{p,q}^*=J_q^*P_{p,q}^*
 \ \ \ \text{and} \ \ \ 
B_{p,q}^*=J_q^*P_{p-1,q},
\]
we have the desired
\[
\Delta_{p,q}
=
A_{p,q}^*A_{p,q}+B_{p,q}^*B_{p,q}
+
\left(p-q-\frac{p^2}{N}\right)\mathrm{Id}.
\]
The formula for $\cQ_{p,q}(a)$ follows by taking the inner product $(\Delta_{p,q}a,a)$.
\end{proof}

\subsection{Computation of cohomology groups}
The Bochner identity implies the following:
\begin{proposition}\label{prop:spectral}
For every $a\in\Omega^q(M;F_p)$,
\begin{equation}\label{eq:base-gap}
 \cQ_{p,q}(a)
 \geq\left(\frac{p(N-p)}N-q\right)\|a\|_M^2.
\end{equation}
\end{proposition}

\begin{proof}
Both squared norms on the right-hand side of
\eqref{eq:compact-squares} are nonnegative. Removing them
gives \eqref{eq:base-gap}.
\end{proof}

The flat connections need not preserve the positive fiber
metrics used in this estimate. We recall the passage from
the corresponding Hodge Laplacian to flat cohomology.

\begin{lemma}\label{lem:flat-Hodge}
Let $F$ be a flat real vector bundle over a closed Riemannian
manifold, equipped with any positive fiber metric. Let $\dd_F$
be its flat exterior derivative and $\dd_F^*$ its formal
adjoint. Every class in $H^q(M;F)$ has a unique representative
in
\[
 \ker\Delta_{F,q},\qquad
 \Delta_{F,q}=\dd_F\dd_F^*+\dd_F^*\dd_F.
\]
\end{lemma}

\begin{proof}
Flatness gives $\dd_F^2=0$. The principal symbol of $\dd_F$
is exterior multiplication by a covector, tensored with
$\id_F$. Its adjoint symbol is contraction by the metric
dual covector. The relation
$\varepsilon(\xi)\iota(\xi^\sharp)
 +\iota(\xi^\sharp)\varepsilon(\xi)=|\xi|^2\id$
shows that the principal symbol for
$\Delta_{F,q}$ is $|\xi|^2\id$. Thus this operator is
elliptic. It is formally self-adjoint and nonnegative, since
\[
 (\Delta_{F,q}a,a)
       =\|\dd_Fa\|^2+\|\dd_F^*a\|^2.
\]
On a closed manifold, self-adjoint elliptic spectral theory
gives its finite-dimensional smooth kernel, its orthogonal
projection $P$ onto $\ker(\Delta_{F,q})$, and a Green operator $G$, satisfying
\[
 \id=P+\Delta_FG=P+G\Delta_F.
\]
These operators preserve smooth forms. From
$\dd_F^2=(\dd_F^*)^2=0$ one obtains, by expanding the products,
$\dd_F\Delta_F=\Delta_F\dd_F$ and the corresponding identity
for $\dd_F^*$. The same commutation holds with $G$. Therefore, if $\dd_Fa=0$, then $\dd_FGa=G\dd_Fa=0$, so
\[
 a=Pa+\Delta_FGa
   =Pa+\dd_F\dd_F^*Ga.
\]
The harmonic form $Pa$ represents the same cohomology class.
If a harmonic form is exact, say $a=\dd_Fb$, then
$\|a\|^2=(\dd_F^*a,b)=0$. This proves uniqueness as well.
\end{proof}

\begin{proof}[Proof of Theorem \ref{thm:vanishing}]
Apply Lemma \ref{lem:flat-Hodge} to $F_p$ with the metrics in
Proposition \ref{prop:spectral}. A harmonic representative $a$ has
$\cQ_{p,q}(a)=0$. If $q<p(N-p)/N$, the coefficient in
\eqref{eq:base-gap} is positive, and hence $a=0$.
Therefore $H^q(M;F_p)=0$.
\end{proof}

\section{The Euler characteristic in dimension four}\label{sec:euler}

We now turn the analytic estimate into the topological result in Theorem
\ref{thm:four}. We recall explicitly a well-known fact about cohomology with flat
coefficients.

\begin{lemma}\label{lem:duality}
Let $M$ be a closed oriented $n$-manifold and let $F$ be a flat real
vector bundle. Integration of the fiberwise evaluation pairing
induces a nondegenerate pairing
\begin{equation}\label{eq:PD}
 H^q(M;F)\times H^{n-q}(M;F^*)\longrightarrow\R.
\end{equation}
Moreover,
\begin{equation}\label{eq:flat-Euler}
 \chi(M;F)=\rank(F)\chi(M).
\end{equation}
\end{lemma}

\begin{proof}[Proof of Theorem \ref{thm:four}]
Assume first that $M$ is oriented. For $n=4$, the standard bundle
$E$ has rank $N=5$. Taking $p=2$ and $p=3$ in
Theorem \ref{thm:vanishing} gives
\begin{equation}\label{eq:low-vanish}
 H^0(M;F_2)=H^1(M;F_2)=0,
 \qquad
 H^0(M;F_3)=H^1(M;F_3)=0,
\end{equation}
because $p(5-p)/5=6/5$ in both cases.

Let $\nu$ be a parallel nonzero volume form on $E$. Contraction
with $\nu$ is a flat bundle isomorphism
\[
 \Lambda^2E\longrightarrow\Lambda^3E^*,
 \qquad v\wedge w\longmapsto\iota(w)\iota(v)\nu.
\]
Consequently $F_2^*\cong F_3$. Applying
Lemma \ref{lem:duality} and \eqref{eq:low-vanish}, we obtain
\[
 H^4(M;F_2)\cong H^0(M;F_3)^*=0,
 \qquad
 H^3(M;F_2)\cong H^1(M;F_3)^*=0.
\]
The only possibly nonzero cohomology group of $F_2$ is therefore
$H^2(M;F_2)$. Since $\rank(F_2)=\binom52=10$,
\eqref{eq:flat-Euler} gives
\[
 10\chi(M)=\chi(M;F_2)
 =\dim_{\R}H^2(M;F_2)\geq0.
\]
This proves both conclusions for oriented $M$.

If $M$ is nonorientable, its orientation double cover
$\widehat M$ is again a closed properly convex projective
four-manifold. The oriented case applies to $\widehat M$.
Euler characteristic multiplies by the degree of a finite cover,
so $2\chi(M)=\chi(\widehat M)\geq0$.
\end{proof}

\subsection{The spectrum of Euler characteristics}

\begin{example}\label{ex:torus}
The inequality is attained by convex projective four-tori.
Take the positive cone $\C=\R_{>0}^5$. Its projectivization is an
open simplex. The determinant-one positive diagonal group is
\[
 A=\{\operatorname{diag}(e^{u_1},\ldots,e^{u_5}):
                     u_1+\cdots+u_5=0\}\cong\R^4.
\]
In logarithmic coordinates, $A$
acts by translations on $\R^4$. A lattice $\Gamma<A$ therefore
acts freely and cocompactly on the simplex, with quotient $T^4$.
Its Euler characteristic is zero.
\end{example}

\begin{corollary}\label{cor:spectrum}
The Euler characteristics of closed oriented properly convex real
projective four-manifolds are exactly the non-negative even integers.
\end{corollary}

\begin{proof}
The splitting $E=\R\sigma\oplus D^E\sigma(TM)$ in
Section \ref{sec:compact-Bochner} gives a smooth bundle isomorphism
$TM\oplus\R\cong E$. Since $E$ is flat, Chern--Weil theory
and stability give $p_1(TM)=0$. This is the present
case of Kobayashi's theorem
\cite[Theorem~4.16]{Kobayashi1984}, based on Avez's identity;
see \cite[Section~2.2]{RioloSeppiSlavich}.
The Hirzebruch signature theorem gives $\sigma(M)=0$.
Writing $b_2=b_2^++b_2^-$, this implies $b_2^+=b_2^-$, and
Poincar\'e duality gives
\[
 \chi(M)=2-2b_1+b_2=2-2b_1+2b_2^+\in2\mathbb Z.
\]
It is nonnegative by Theorem \ref{thm:four}.
Conversely, zero is realized in Example \ref{ex:torus}, and every
positive even integer is realized by
\cite[Theorem 1.3]{RioloSeppiSlavich}.
\end{proof}

\appendix
\section{Coordinate proof of Akagawa's identities}\label{app:Hessian}

We prove \eqref{eq:Kahler} for a positive Hessian metric
$g=D\dd\phi$. The argument is local and uses only the first
derivatives of $g$. In particular, neither the canonical
Monge--Amp\`ere equation, nor completeness
is assumed. 

\subsection{Metric derivatives and the Levi--Civita connection}

Choose affine coordinates $x^1,\ldots,x^N$ and put
$D_i=D_{\partial_i}$. Thus $D_i\dd x^j=0$, and the covariant
components of $g$ and its cubic derivative are
\[
 g_{ij}=\phi_{ij},\qquad
 C_{ijk}=\tfrac12\phi_{ijk}.
\]
The tensor $C$ is totally symmetric. We reserve $\nabla^g$ for
the Levi--Civita connection, to distinguish it from $D$.
Since every first derivative of $g$ is a third derivative of
$\phi$, its Christoffel symbols simplify to
\begin{align}
 \Gamma^k_{ij}
 &=\tfrac12 g^{k\ell}
   (\partial_i g_{j\ell}+\partial_j g_{i\ell}
                     -\partial_\ell g_{ij})
   =\tfrac12 g^{k\ell}\phi_{ij\ell},
 \qquad
 \alpha_i=\sum_j\Gamma^j_{ij}
          =\tfrac12 g^{ab}\phi_{iab}.
 \label{eq:app-Christoffel}
\end{align}
Thus $\nabla^g_{\partial_i}\partial_j
=\Gamma^k_{ij}\partial_k$, whereas
$D_{\partial_i}\partial_j=0$.

\noindent Differentiating $g^{ia}g_{aj}=\delta^i_j$ gives
\[
 \partial_k g^{ij}
 =-g^{ia}(\partial_k g_{ab})g^{bj}
 =-2g^{ia}g^{jb}C_{kab}.
\]
For the density $\mu=\sqrt{\det(g_{ij})}$, Jacobi's determinant
formula gives
\[
 \partial_i\mu
 =\tfrac12\mu\,g^{ab}\partial_i g_{ab}
 =\alpha_i\mu.
\]

Fix a point $x_0$. A constant linear change of affine coordinates
makes $g_{ij}(x_0)=\delta_{ij}$. These remain affine coordinates;
they are not Riemannian normal coordinates, and the derivatives
of $g$ must still be retained. At this point,
\begin{equation}\label{eq:app-metric-derivative}
 \Gamma^k_{ij}=C_{ijk},\qquad
 \partial_k g^{ij}=-2C_{kij},\qquad
 \alpha_i=\sum_j C_{ijj},\qquad
 \sum_i\partial_i g^{ij}=-2\alpha_j.
\end{equation}
The last equality uses the total symmetry of $C$.

\subsection{Exterior operators and differentiation of the tensor metric}

As in Section \ref{sec:double}, write
\[
 e_i=\varepsilon(\dd x^i),\quad
 \bar e_i=\bar\varepsilon(\dd x^i),\quad
 i_i=\iota(\partial_i),\quad
 \bar i_i=\bar\iota(\partial_i).
\]
Their relations are
\begin{equation}\label{eq:app-algebra}
 i_ie_j+e_ji_i=\delta_{ij}\id,
 \qquad
 \bar i_i\bar e_j+\bar e_j\bar i_i=\delta_{ij}\id.
\end{equation}
Two contractions in the same factor anticommute, and all
operators in different factors commute. In particular,
\begin{equation}\label{eq:app-operators}
 \partial=\sum_i e_iD_i,\qquad
 \bdel=\sum_i\bar e_iD_i,\qquad
 \Lambda=\sum_{i,j}g^{ij}i_i\bar i_j.
\end{equation}

\noindent On double forms of a fixed bidegree, let
\[
 S_i=\sum_{j,k}\Gamma^k_{ij}
       (e_ji_k+\bar e_j\bar i_k).
\]
This is the connection action on every covariant slot:
\begin{align*}
 (S_i\omega)_{I;J}
 ={}&\sum_{r=1}^p\Gamma^a_{i i_r}
    \omega_{i_1\ldots i_{r-1}a i_{r+1}\ldots i_p;J}\\
 &+\sum_{s=1}^q\Gamma^b_{i j_s}
    \omega_{I;j_1\ldots j_{s-1}b j_{s+1}\ldots j_q}.
\end{align*}
Consequently $\nabla_i^g\omega=D_i\omega-S_i\omega$.
At $x_0$, its expression is
\[
 S_i=\sum_{j,k}C_{ijk}(e_ji_k+\bar e_j\bar i_k).
\]
The matrix $\Gamma_i$ is self-adjoint for the relevant metric:
lowering its upper index gives
$g_{k\ell}\Gamma^\ell_{ij}=C_{ijk}$, which is symmetric in
$j,k$. Its induced action $S_i$ on exterior tensors is therefore
self-adjoint. Since $\nabla^g$ preserves their metric,
\begin{align}
 \partial_i\langle\omega,\eta\rangle_g
 &=\langle\nabla_i^g\omega,\eta\rangle_g
   +\langle\omega,\nabla_i^g\eta\rangle_g\notag\\
 &=\langle D_i\omega,\eta\rangle_g
   +\langle\omega,D_i\eta\rangle_g
   -2\langle\omega,S_i\eta\rangle_g.
 \label{eq:app-inner-derivative}
\end{align}

\subsection{The formal adjoints}

Take smooth test double forms supported in the affine chart.
The integral of the ordinary derivative
$\partial_i(\mu\langle\omega,\eta\rangle_g)$ is zero.
Using $\partial_i\mu=\alpha_i\mu$ and
\eqref{eq:app-inner-derivative}, we get
\begin{align*}
 0=\int
 \bigl(&\langle D_i\omega,\eta\rangle_g
       +\langle\omega,D_i\eta\rangle_g
       -2\langle\omega,S_i\eta\rangle_g
       +\alpha_i\langle\omega,\eta\rangle_g\bigr)
       \,\dd\vol_g.
\end{align*}
Solving for the first term gives the exact local formal adjoint
\begin{equation}\label{eq:app-D-adjoint}
 D_i^*=-D_i-\alpha_i+2S_i.
\end{equation}

\noindent Pointwise wedge multiplication has adjoints
\[
 e_i^*=g^{ij}i_j,\qquad
 \bar e_i^*=g^{ij}\bar i_j.
\]
It follows that, on input of bidegree $(p,q)$,
\[
 \bdelt=\sum_{i,j}D_i^*(g^{ij}\bar i_j).
\]
Here $g^{ij}\bar i_j$ first maps to bidegree $(p,q-1)$;
the $D_i^*$ that follows it acts on that
lower bidegree. Expanding the derivative and then setting
$g_{ij}(x_0)=\delta_{ij}$ gives
\begin{align*}
 \bdelt
 ={}&-\sum_{i,j}g^{ij}\bar i_jD_i
     -\sum_{i,j}(\partial_i g^{ij})\bar i_j
     -\sum_{i,j}\alpha_i g^{ij}\bar i_j
     +2\sum_{i,j}g^{ij}S_i\bar i_j\\
 ={}&-\sum_i\bar i_iD_i
     +2\sum_i\alpha_i\bar i_i
     -\sum_i\alpha_i\bar i_i
     +2\sum_iS_i\bar i_i
     \qquad\text{at }x_0.
\end{align*}

\noindent There are two cubic contributions in $\sum_iS_i\bar i_i$.
The one acting only in the second factor vanishes:
\[
 T:=\sum_{i,j,k}C_{ijk}\bar e_j\bar i_k\bar i_i=0.
\]
Indeed, exchanging the dummy indices $i,k$ leaves $C_{ijk}$
unchanged and replaces $\bar i_k\bar i_i$ by
$\bar i_i\bar i_k=-\bar i_k\bar i_i$, so $T=-T$.
The mixed contribution remains. We obtain
\begin{equation}\label{eq:app-bar-adjoint}
 \bdelt=-\sum_i\bar i_iD_i+\sum_i\alpha_i\bar i_i
          +2\sum_{i,j,k}C_{ijk}e_ji_k\bar i_i
          \qquad\text{at }x_0.
\end{equation}
Interchanging the two exterior factors gives
\begin{equation}\label{eq:app-first-adjoint}
 \delta=-\sum_i i_iD_i+\sum_i\alpha_i i_i
          +2\sum_{i,j,k}C_{ijk}\bar e_j\bar i_ki_i
          \qquad\text{at }x_0.
\end{equation}
For this second formula, $S_i$ acts after contraction on bidegree
$(p-1,q)$, and its pure first-factor contribution vanishes by
the same symmetry argument.

\subsection{Commuting contraction past the exterior derivatives}

We now compute $\Lambda\partial+\partial\Lambda$ at $x_0$.
The coefficients $g^{ij}$ in $\Lambda$ must again be
differentiated before evaluating there. Using
\eqref{eq:app-operators},
\begin{align*}
 \Lambda\partial+\partial\Lambda
 ={}&\sum_{i,j,k}g^{ij}
       (i_ie_k+e_ki_i)\bar i_jD_k
       +\sum_{i,j,k}(\partial_kg^{ij})e_ki_i\bar i_j\\
 ={}&\sum_i\bar i_iD_i
       -2\sum_{i,j,k}C_{kij}e_ki_i\bar i_j.
\end{align*}
In the first line, operators in different factors commute, and
$D_k$ commutes with contractions in the affine frame. The
exterior relation $i_ie_k+e_ki_i=\delta_{ik}\id$ gives the
first term in the second line. Total symmetry of $C$ and a
renaming of indices rewrite its cubic term as
\[
 -2\sum_{i,j,k}C_{ijk}e_ji_k\bar i_i.
\]
Comparison with \eqref{eq:app-bar-adjoint} therefore gives
\[
 \Lambda\partial+\partial\Lambda
 =-\bdelt+\sum_i\alpha_i\bar i_i
 =-\bdelt+\bar\iota(X_\alpha).
\]
At $x_0$, $X_\alpha=\sum_i\alpha_i\partial_i$, which explains
the last expression. The mixed cubic terms agree exactly and
cancel; only the trace $\alpha$ remains.

For the other factor, the same expansion gives
\begin{align*}
 \Lambda\bdel+\bdel\Lambda
 ={}&\sum_{i,j,k}g^{ij}i_i
       (\bar i_j\bar e_k+\bar e_k\bar i_j)D_k
       +\sum_{i,j,k}(\partial_kg^{ij})\bar e_ki_i\bar i_j\\
 ={}&\sum_i i_iD_i
       -2\sum_{i,j,k}C_{kij}\bar e_ki_i\bar i_j\\
 ={}&\sum_i i_iD_i
       -2\sum_{i,j,k}C_{ijk}\bar e_j\bar i_ki_i\\
 ={}&-\delta+\iota(X_\alpha),
\end{align*}
where cross-factor contractions commute and the final comparison
uses \eqref{eq:app-first-adjoint}. This proves both identities
in \eqref{eq:Kahler}. Since $x_0$ was arbitrary and the operators
are intrinsic, the identities hold on the entire Hessian manifold.

\end{document}